\documentclass[12pt,oneside,reqno]{amsart}

\usepackage{amssymb,latexsym,amsxtra,amscd,ifthen}

\usepackage{amsmath}

\usepackage{amsfonts}

\usepackage{amsthm}

\usepackage{verbatim}

\usepackage{dsfont}

\usepackage{mathtools}

\usepackage{enumerate}

\usepackage[capitalise]{cleveref}

\usepackage{color}

\usepackage{endnotes}

\usepackage{tikz-cd}

\newtheorem{innercustomgeneric}{\customgenericname}
\providecommand{\customgenericname}{}
\newcommand{\newcustomtheorem}[2]{%
  \newenvironment{#1}[1]
  {%
   \renewcommand\customgenericname{#2}%
   \renewcommand\theinnercustomgeneric{##1}%
   \innercustomgeneric
  }
  {\endinnercustomgeneric}
}

\theoremstyle{plain}

\newtheorem{theorem}{Theorem}

\newtheorem{lemma}{Lemma}

\newtheorem*{corollary*}{Corollary}

\newtheorem*{theorem*}{Theorem}

\newcustomtheorem{customthm}{Theorem}

\newtheorem*{proposition*}{Theorem}

\newtheorem*{question*}{Question}

\newtheorem*{conjecture*}{Conjecture}

\theoremstyle{definition}

\theoremstyle{remark}

\DeclareMathAlphabet{\mathbbold}{U}{bbold}{m}{n}
\def\bb1{\mathbbold{1}}

\def\bbq{\mathbb{Q}}

\def\bbc{\mathbb{C}}

\DeclareMathOperator\Aut{Aut}

\DeclareMathOperator\ad{ad}

\newtheorem*{remark*}{Remark}

\usepackage{xparse}
\DeclareDocumentCommand{\adx}{ O{2} O{x_1}  }{\ad_{#2}^{[#1]}}

\begin{document}
\title{On the Morita equivalence class of a finitely presented algebra}
\author{Adel Alahmadi\protect\endnotemark[1], Hamed Alsulami\protect\endnotemark[1], Efim Zelmanov\protect\endnotemark[2]$^{,1}$}
%\dedicatory{To Boris Isaakovich Plotkin with admiration.}
	
	%\keywords{growth function, associative algebra, wreath product}
	
	%\subjclass[2010]{Primary: 20E18, 20F50, 20F40, 16R99}

	
	\dedicatory{To our teacher and friend S.~K.~Jain.}

\address{1. To whom correspondence should be addressed\hfill\break
E-mail: ezelmano@math.ucsd.edu\hfill\break
Author Contributions: A.A., H.A, E. Z. designed research, performed research, and wrote the paper. The authors declare no conflict of interest.}
	
%\begin{abstract}
%We introduce a new construction of matrix wreath products of algebras that is similar to wreath products of groups. We then use it to prove embedding theorems for Jacobson radical, nil, and primitive algebras. In \S\ref{Section6}, we construct finitely generated nil algebras of arbitrary Gelfand-Kirillov dimension $\geq 8$ over a countable field which answers a question from \cite{7}.
%\end{abstract}

\maketitle
Let $F$ be a field and let $A,B$ be $F$-algebras. The algebras $A,B$ are Morita equivalent if their categories of (left) modules are equivalent. Throughout the paper we consider algebras with $1$ and unital modules over them. Then $A$ is Morita equivalent to $B$ if and only if $A \cong eM_n(B)e,$ where $M_n(B)$ is the algebra of $n \times n$ matrices over $B$ and $e = e^2 \in M_n(B)$ is a full idempotent, i.e. $M_n(B)eM_n(B) = M_n(B)$.

In [2], [3], [4] the authors described Morita equivalence classes of some important finitely presented algebras and, in particular, showed that these classes are countable.

In this note we discuss Morita equivalence classes of arbitrary finitely presented algebras. Let $\sigma$ be an automorphism of the field $F$. The mapping $\varphi: A \to B$ is called a $\sigma$-semilinear isomorphism if $\varphi$ is a ring isomorphism and $\varphi(\alpha a) = \sigma(\alpha)\varphi(a)$ for all $\alpha \in F$, $a \in A$.

Let $V$ be a vector space over the field $F$. On the abelian group $V$ we define a new multiplication by scalars: $\alpha \cdot v = \sigma(\alpha)v$. We will denote the new vector space as $V^{(\sigma)}$.

For an $F$-algebra $A$ define a new $F$-algebra structure $A^{(\sigma)}$ on $A$ by keeping the ring structure and switching to the vector space $A^{(\sigma)}$ instead of $A$.

It is easy to see that the identical mapping $A \to A^{(\sigma^{-1})}$ is a $\sigma$-semilinear isomorphism.

Suppose that an $F$-algebra $A$ is generated by a finite collection of elements $a_1,\ldots,a_m$. Consider the free associative algebra $F\langle x_1,\ldots,x_m \rangle$ and the homomorphism
\[
	F\langle x_1,\ldots,x_m  \rangle \overset{\varphi}{\longrightarrow} A, \ x_i \mapsto a_i, \ 1 \leq i \leq m.
\]
Let $R$ be a subset of the ideal $I = \ker \varphi$ that generates $I$ as an ideal. We say that the algebra $A$ has presentation
\[
	A = \langle x_1,\ldots,x_m \mid R = 0\rangle.
\]
If the set $R$ is finite then the algebra $A$ is said to be finitely presented. This property does not depend on a choice of a generating system as long as the system is finite.

It is easy to see that semilinearly isomorphic algebras are Morita equivalent. Hence the group $\Aut F$ acts on the class of Morita equivalence of the algebra $A$.

\begin{theorem}\label{Theorem1}
Let $F$ be an algebraically closed field and let $A$ be a finitely presented $F$-algebra. Then the action of $\Aut F$ on the class of Morita equivalence of $A$ has countably many orbits.
\end{theorem}

In other words, the Morita equivalence class is countable up to semilinear isomorphisms. The following observation is straightforward.

\begin{lemma}\label{Lemma1}
Let
\begin{align*}
	A = &\langle x_1,\ldots,x_m \mid \sum\nolimits_j{\alpha_{ij}w_{ij} = 0}, \ 1 \leq i \leq n, \ \alpha_{ij} \in F, \\
	&w_{ij} \text{ are words in } x_1,\ldots,x_m\rangle.
\end{align*}
Then
\[
	A^{(\sigma)} = \langle x_1,\ldots,x_m \mid \sum_j{\sigma^{-1}(\alpha_{ij})w_{ij}} = 0, \ 1 \leq i \leq n\rangle.
\]
\end{lemma}

In [1] it was shown that finite presentation is a Morita invariant property. In other words, the Morita equivalence class of a finitely presented $F$-algebra $A$ consists of finitely presented $F$-algebras. If the field $F$ is countable then it immediately follows that the Morita equivalence class of $A$ is countable up to isomorphisms.

Assume therefore that the field $F$ is not countable.
\begin{proof}[Proof of Theorem \ref{Theorem1}]
Let $F_0$ be the prime subfield of $F$. Let $X \subset F$ be a maximal subset of $F$ that is algebraically independent over $F_0$. Then $F$ is the algebraic closure of the purely transcendental extension $F_0(X)$. Since the field $F$ is uncountable it follows that the set $X$ is uncountable as well.

Let $X_0$ be a countable subset of $X$. Let $F_0(X_0)$ be the purely transcendental $F_0$-extension generated by $X_0$ and let $\widetilde{F_0(X_0)}$ be the algebraic closure of $F_0(X_0)$ in $F$.

We claim that for any finite collection of elements $\alpha_1, \ldots, \alpha_n \in F$ there exists an automorphism $\sigma \in \Aut F$ that maps $\alpha_1,\ldots,\alpha_n$ to $\widetilde{F_0(X_0)}$. Indeed, there exists a finite subset $X' \subset X$ such that $\alpha_1,\ldots,\alpha_n \in \widetilde{F_0(X')}$. Let $p: X \to X$ be a bijection such that $p(X') \subset X_0$. The bijection $p$ extends to an automorphism $\sigma \in \Aut F$. Clearly, $\sigma(\alpha_i) \in \widetilde{F_0(X_0)}$.

We say that an algebra is finitely presented over a subfield $K \subset F$ if it has a presentation
\[
	\langle x_1,\ldots,x_m \mid R = 0\rangle, \ |R| < \infty, \ R \subset K\langle x_1,\ldots,x_m \rangle.
\]
If $B$ is a finitely presented $F$-algebra then there exists an automorphism $\tau \in \Aut F$ such that $B^{(\tau)}$ is finitely presented over $\widetilde{F_0(X_0)}$.

Indeed, let
\begin{align*}
	B = &\langle x_1,\ldots,x_m \mid \sum\nolimits_j{\alpha_{ij}w_{ij} = 0}, \ 1 \leq i \leq n, \ \alpha_{ij} \in F, \\
	&w_{ij} \text{ are words in } x_1,\ldots,x_m\rangle.
\end{align*}
There exists an automorphism $\sigma \in \Aut F$ such that $\sigma(\alpha_{ij}) \in \widetilde{F_0(X_0)}$ for all $i,j$. By Lemma \ref{Lemma1} it implies that $B^{(\sigma^{-1})}$ is finitely presented over $\widetilde{F_0(X_0)}$.

An arbitrary algebra $B$ that is Morita equivalent to $A$ is finitely presented [1]. Hence there exists $\tau \in \Aut F$ such that $B^{(\tau)}$ is finitely presented over $\widetilde{F_0(X_0)}$. Since the field $\widetilde{F_0(X_0)}$ is countable it completes the proof of Theorem \ref{Theorem1}.
\end{proof}

Now we will show that, generally speaking, the class of Morita equivalence of a finitely presented algebra may be uncountable up to isomorphisms.

\begin{lemma}\label{Lemma2}
Let $\alpha \in F$. Consider the algebra
\[
	A_{\alpha} = \langle x_1, x_2 \mid x_1^2 + x_2^2 + \alpha x_1x_2 = 0\rangle.
\]
Then $A_\alpha \cong A_\beta$ if and only if $\beta = \pm \alpha$.
\end{lemma}
\begin{proof}
Clearly, $A_\alpha \cong A_{-\alpha}$. Suppose now that $\beta \neq \pm \alpha$. We will show that the algebra $A_\alpha$ does not contain generators $y_1, y_2$ such that $y_1^2 + y_2^2 + \beta y_1y_2 = 0$.

Suppose the contrary and let
\begin{align*}
y_1 &= \alpha_{10} + \alpha_{11}x_1 + \alpha_{12}x_2 + y_1' \\
y_2 &= \alpha_{20} + \alpha_{21}x_1 + \alpha_{22}x_2 + y_2'
\end{align*}
be such generating elements; $\alpha_{ij} \in F$; $y_1', y_2'$ are linear combinations of monomials of length $\geq 2$. The matrix
\[
	Q = \begin{pmatrix}	\alpha_{11} & \alpha_{12} \\ \alpha_{21} & \alpha_{22} \end{pmatrix}
\]
is nonsingular, otherwise the elements $y_1, y_2$ do not generate $A\alpha$.

The linear part of $y_1^2 + y_2^2 + \beta y_1y_2$ is equal to
\begin{align*}
	&2 \alpha_{10}(\alpha_{11}x_1 + \alpha_{12}x_2) + 2 \alpha_{20}(\alpha_{21}x_1 + \alpha_{22} x_2) + \beta \alpha_{10}(\alpha_{21}x_1 + \alpha_{22}x_2) \\
	&\hspace{1cm} +\beta \alpha_{20}(\alpha_{11}x_1 + \alpha_{12}x_2) \\
	&= (2 \alpha_{10}\alpha_{11} + 2 \alpha_{20}\alpha_{21} + \beta \alpha_{10}\alpha_{21} + \beta \alpha_{20}\alpha_{11}) x_1 \\
	&\hspace{1cm} +(2 \alpha_{10}\alpha_{12} + 2\alpha_{20} \alpha_{22} + \beta \alpha_{10} \alpha_{22} + \beta \alpha_{20} \alpha_{21})x_2\\
	&= 0.
\end{align*}
Hence
\begin{align*}
	\alpha_{10}(2\alpha_{11} + \beta \alpha_{21}) + \alpha_{20}(2 \alpha_{21} + \beta \alpha_{11}) &= 0, \\
	\alpha_{10}(2\alpha_{12} + \beta \alpha_{22}) + \alpha_{20}(2 \alpha_{22} + \beta \alpha_{12}) &= 0.	
\end{align*}
We have
\[
	\begin{pmatrix} 2 \alpha_{11} + \beta \alpha_{21} & 2 \alpha_{21} + \beta \alpha_{11} \\ 2 \alpha_{12} + \beta \alpha_{22} & 2 \alpha_{22} + \beta \alpha_{12} \end{pmatrix} = \begin{pmatrix} 2 & \beta \\ \beta & 2 \end{pmatrix} \begin{pmatrix} \alpha_{11} & \alpha_{12} \\ \alpha_{21} & \alpha_{22} \end{pmatrix}.
\]
If $\beta \neq \pm 2$ then this matrix is nonsingular, which implies $\alpha_{10} = \alpha_{20} = 0$.

Suppose that this is the case, i.e. $\beta \neq \pm 2$, $\alpha_{10} = \alpha_{20} = 0$.

The homogeneous component of degree $2$ of $y_1^2 + y_2^2 + \beta y_1y_2$ is equal to
\[
	(\alpha_{11}x_1 + \alpha_{12}x_2)^2 + (\alpha_{21}x_1 + \alpha_{22}x_2)^2 + \beta(\alpha_{11}x_1 + \alpha_{12}x_2)(\alpha_{21}x_1 + \alpha_{22}x_2) = 0.
\]
Let $x = (x_1,x_2)$. Then the left hand side is equal to $xQ^T\begin{pmatrix} 1 & \beta \\ 0 & 1 \end{pmatrix} Qx^T$. Since this expression is equal to zero in the algebra
\[
\langle x_1,x_2 \mid x_1^2 + x_2^2 + \alpha x_1 x_2 = 0 \rangle
\]
it follows that in the free algebra it is equal to $\gamma(x_1^2 + x_2^2 + \alpha x_1 x_2), \gamma \in F$. Hence,
\[
	Q^T \begin{pmatrix} 1 & \beta \\ 0 & 1 \end{pmatrix} Q = \gamma \begin{pmatrix} 1 & \alpha \\ 0 & 1 \end{pmatrix}.
\]
It is well-known in the theory of quadratic forms that this equality implies $\beta = \pm \alpha$.

We proved that $\beta = \pm 2$ or $\alpha = \pm \beta$. Similarly, $\alpha = \pm 2$ or $\beta = \pm \alpha$. Thus, if $\alpha \neq \pm \beta$ then $\alpha = \pm 2$ and $\beta = \pm 2$ which again implies that $\alpha = \pm \beta$. This completes the proof of the lemma.
\end{proof}

\vspace{12pt}
\begin{corollary*}
If a complex number $\alpha \in \bbc$ is not algebraic over $\bbq$ then the Morita equivalence class of the algebra $A_\alpha$ is uncountable.
\end{corollary*}
Indeed, if $\alpha$ is not algebraic over $\bbq$ then the orbit $O = \{\sigma(\alpha), \alpha \in \Aut \bbc\}$ is uncountable. The system $\{A_\beta, \ \beta \in O\}$ lies in the Morita equivalence class of $A_\alpha$ and is uncountable by Lemma \ref{Lemma1}.

\vspace{12pt}
\begin{question*}
Let $A$ be a $\bbc$-algebra that is finitely presented over $\bbq$. Is it true that the Morita equivalence class of $A$ is countable?
\end{question*}

\section*{Acknowledgements}
The authors are grateful to F. Eshmatov for helpful discussions.

\endnotetext[1]{Department of Mathematics, King Abdulaziz University, Jeddah, SA,\\
E-mail address: analahmadi@kau.edu.sa; hhaalsalmi@kau.edu.sa;}

\endnotetext[2]{Department of Mathematics, University of California, San Diego, USA\\
E-mail address: ezelmano@math.ucsd.edu}

\theendnotes

\end{document}